\documentclass[letterpaper, 10 pt, conference]{ieeeconf}  

\IEEEoverridecommandlockouts                              

\usepackage{amsmath,amssymb,amsfonts,amsthm,dsfont} 
\usepackage{graphics}
\usepackage[dvips]{graphicx}
\usepackage[table,usenames,dvipsnames]{xcolor}
\usepackage{cancel}
\usepackage{subcaption}
\usepackage{tikz-cd}
\usepackage{algorithm2e}
\usepackage[breaklinks=true, colorlinks, bookmarks=true, citecolor=Black, urlcolor=Violet,linkcolor=Black]{hyperref}

\newtheorem{theorem}{Theorem}
\newtheorem{lemma}{Lemma}

\theoremstyle{definition}

\newcommand{\ba}{\begin{align}}
\newcommand{\ea}{\end{align}}
\newcommand{\fr}{\frac}

\newcommand{\lam}{\lambda}
\newcommand{\R}{{\mathbb R}}

\title{\LARGE \bf Neuron Growth Control by PDE Backstepping:\\
Axon Length Regulation by Tubulin Flux Actuation in Soma}

\author{Cenk Demir \and Shumon Koga \and Miroslav Krstic
\thanks{C. Demir and M. Krstic are with the Department of Mechanical and Aerospace Engineering, U.C. San Diego, 9500 Gilman Drive, La Jolla, CA, 92093-0411, {\tt\small cdemir@ucsd.edu} and {\tt\small krstic@ucsd.edu}.}
\thanks{S. Koga is with the Department of Electrical and Computer Engineering, UC San Diego, 9500 Gilman Drive, La Jolla, CA, 92093-0411, {\tt\small skoga@ucsd.edu}.}
}

\begin{document}

\maketitle
\thispagestyle{empty}
\pagestyle{empty}


\begin{abstract}
In this work, stabilization of an axonal growth in a neuron associated with the dynamics of tubulin concentration is proposed by designing a boundary control. The  dynamics are given by a parabolic Partial Differential Equation (PDE) of the tubulin concentration, with a spatial domain of the axon's length governed by an Ordinary Differential Equation (ODE) coupled with the tubulin concentration in the growth cone. We propose a novel backstepping method for the coupled PDE-ODE dynamics with a moving boundary, and design a control law for the tubulin concentration flux in the soma. Through employing the Lyapunov analysis to a nonlinear target system, we prove a local exponential stability of the closed-loop system under the proposed control law in the spatial $H_1$-norm. 
\end{abstract}

\section{Introduction}
Neuroscience has become one of the most significant areas in biology, as scientists pursue the understanding of the functionality of perception and the brain, and investigate therapeutics for  neurological diseases and injuries (see \cite{ribar2020neuromorphic, 7402491}). Neurological disorders, such as Parkinson's disease, spinal cord injuries, and Alzheimer's disease, result from a loss of function of neurons \cite{maccioni2001molecular}, \cite{dauer2003parkinson}, \cite{liu1997neuronal}. Particular medical therapeutics, i.e. Chondroitinase ABC (ChABC), may restore the functionality of neurons (see \cite{karimi2010synergistic, bradbury2011manipulating}) by manipulating the extracellular matrix (ECM), the area surrounding the neuron which contains various macromolecules and minerals that facilitate cell processes \cite{frantz2010extracellular}.  

Neurons are cells that are specialized to obtain perception by transmitting electrical signals along their axons
. These signals enter from the dendrites, travel through the axon, and transmit through the growth cone to another neuron as shown in Figure \ref{fig:my_label}. The presence of chemical cues surrounding the growth cone attract or repel the axon toward another neuron \cite{DIEHL2014194}. When the chemical cues attract the growth cone, tubulin proteins assemble from free tubulin dimers and create microtubules which extend the axon toward the other neuron \cite{julien1999neurofilament}. 
The formation of microtubules is determined by  the size and dynamics of tubulin concentration in the neuron and is supported by ECM \cite{DENT2003209, barros2011extracellular}. The dynamics of tubulin depend on the tubulin production rate in the soma, the degradation rate, the assembly rate, and the transportation process \cite{DIEHL2014194}. According to recent research, it is also possible to control axon elongation, namely tubulin concentration, by manipulating ECM \cite{burnside2014manipulating}.



The dynamics of axon growth have been described by several different models using various mathematical tools. One of the pioneer axonal growth models has been proposed in \cite{van1994neuritic} by considering the transportation of tubulin as diffusion. 
A continuum model of axon growth dynamics has been proposed in \cite{mclean2004continuum}, and the stability analysis for the proposed model is studied in \cite{mclean2006stability}. In another work, the tubulin concentration is modelled by a PDE with a moving boundary \cite{graham2006mathematical, diehl2014one}.  \cite{diehl2016efficient} provides a numerical solution for this moving boundary PDE model. 



While PDEs have been utilized for the computational modeling of axon growth, stabilization of axon growth by means of control theory has not been studied so far. Boundary control of PDEs has been intensively developed by the method of ``backstepping" in the last few decades for various kinds of systems \cite{krstic2008boundary}. One of the initial contributions was achieved in \cite{smyshlyaev2004closed} by applying a Volterra type of transformation to parabolic PDEs by utilizing the method of successive approximations. Following the procedure, the class of the system has been extended to a cascade and coupled PDE-ODE system, see \cite{krstic2009compensating, susto2010control, tang2011state}. While most literature on backstepping design of PDEs has focused on a system with a constant size of domain in time, recently the method has been successfully applied to the Stefan problem, which is a special class of a parabolic PDE with a moving boundary, see \cite{koga2018control, krstic2020materials}. For the model related to axon growth mentioned earlier, the backstepping method for the Stefan problem has been designed and applied to a screw extrusion process of a polymer 3-D printing \cite{koga2020stabilization}. The results mentioned above have been proven to achieve global stability by virtue of the monotonicity of the moving boundary. 

Several researchers have tackled stability analysis for nonlinear PDE systems under the backstepping design of a boundary control by restricting the region of attraction in a local sense. For instance, \cite{coron2013local} proved a local exponential stability of a $2 \times 2$ quasi-linear hyperbolic PDE under the backstepping design by analyzing a Lyapunov function of the spatial $H_2$-norm. As a class of PDEs with a moving boundary, \cite{buisson2018control}  designed a backstepping control law for $2 \times 2$ hyperbolic PDEs with a moving boundary governed by an ODE modeling a piston movement, and proved a local exponential stability in the spatial $H_1$-norm by applying the Lyapunov method to the target system. A similar approach has been done for a moving boundary hyperbolic PDE modeling a shock-wave arising in traffic congestion \cite{yu2020bilateral}. However, those results for local stability analysis have been achieved only for hyperbolic PDE systems, even though the axon growth model proposed in \cite{diehl2014one} is a nonlinear parabolic PDE system. 

In this paper, we develop a boundary control for a coupled PDE-ODE system with a moving boundary which models  the dynamics of tubulin concentration and axon growth. First, we present a steady-state solution of the tubulin concentration for a given constant axon length, and obtain a reference error system to be stabilized at zero states. Next, we apply linearization to the reference error system to deal with the algebraic nonlinearity. 
Then, a backstepping transformation is employed to the linearized reference error dynamics. By solving the gain kernel equations that are derived from backstepping transformation, the control law is obtained. Finally, we prove local exponential stability by applying the Lyapunov method to the nonlinear target system, which ensures the local stability of the original PDE-ODE system of the axon growth model.

This paper is structured as follows. Section \ref{sec:model} introduces the PDE-ODE model of axon growth and tubulin concentration with the steady-state solution. Section \ref{sec:control} presents the control design by the method of backstepping, and the stability result and its proof are given in Section \ref{sec:stability}. The paper ends with the conclusion in Section \ref{sec:conclusion}.  

 \begin{figure}
     \centering
     \includegraphics[scale=0.4]{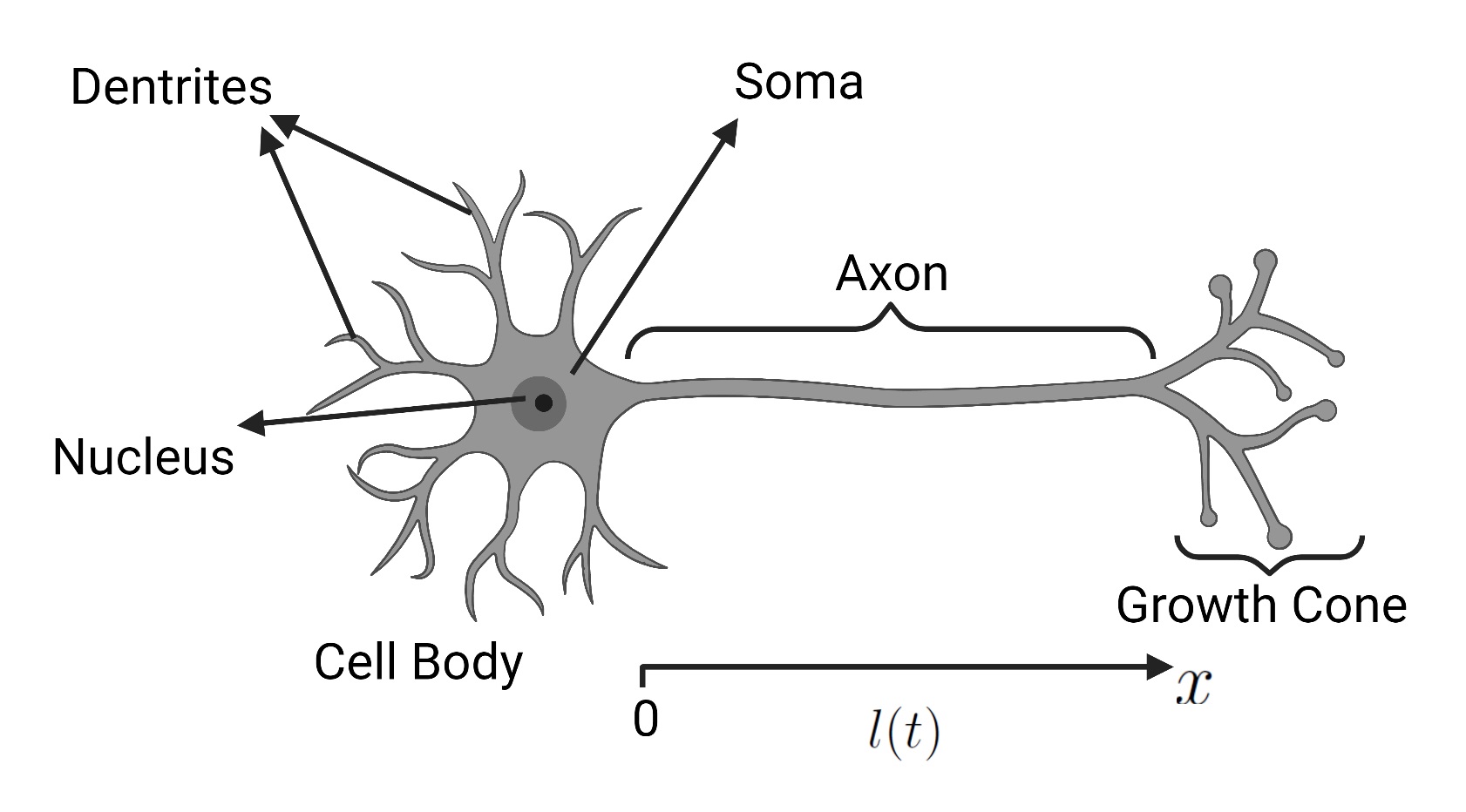}
     \caption{Neuron Structure and PDE domain}
     \label{fig:my_label}
 \end{figure}


\section{Modeling of Axon Growth} \label{sec:model} 
In this section, we present a mathematical model of axon growth governed by a moving boundary PDE, derive a steady-state solution for a given set point of the axon length, and provide a reference error system to be stabilized. 
\subsection{Axon growth model by a moving boundary PDE} 
Tubulin is a group of proteins which is responsible for the growth of a newly created axon. Two assumptions can be described to model this responsibility. First, tubulin proteins are modeled as a homogeneous continuum because free tubulin molecules are very small. Then, tubulin molecules are assumed to be the only factor responsible for axonal growth.
With these assumptions, as proposed in \cite{diehl2014one, diehl2016efficient}, the axonal growth of a newborn axon by tubulin can be modelled as

\begin{align}
\label{sys1} 
c_t (x,t) =& D c_{xx} (x,t) - a c_x (x,t) - g c(x,t) , \\
\label{sys2} c_x(0,t) = & - q_{\rm s}(t), \\
\label{sys3} c(l(t),t) =& c_{\rm c} (t), \\
\label{sys4} l_{\rm c} \dot{c}_{\rm c}(t) = & (a-gl_{\rm c}) c_{\rm c}(t) - D c_x (l(t), t) \notag\\
& - (r_{\rm g} c_{\rm c}(t) + \tilde{r}_{\rm g} l_{\rm c} )(c_{\rm c}(t) - c_{\infty}), \\
\label{sys5} \dot{l}(t) =& r_{\rm g} (c_c(t)-c_{\infty}),
\end{align}
where the tubulin concentration in the axon is denoted as $c(x,t)$, and the variables with subscripts $t$ and $x$ denote the partial derivatives with respect to the subscripts. The variable $q$ denotes the concentration flux. Subscript ``s" is used for the soma of the neuron, and subscript ``c" is used for the cone of the neuron. 
Namely, $q_{\rm s}(t)$ is the concentration flux of tubulin in the soma, and $c_{\rm c}(t)$ is the concentration of tubulin in the cone. The length of the axon in $x$-coordinate is $l(t)$. As time passes, tubulin in the neuron degrade at the constant rate, $g$. $D$ is the diffusivity constant, and $a$ is the velocity constant of tubulin proteins in \eqref{sys1}. $l_{\rm c}$ is the growth ratio of the cone, and $\tilde{r}_{\rm g}$ is the reaction rate to create microtubules. $\tilde{s}_{\rm g}$ is the disassemble rate which means it is the transformation rate from microbules to tubulin dimers and $r_{\rm g}$ is a lumped parameter. The equilibrium of  the tubulin concentration in the cone, which is denoted as $c_{\infty}$, causes the axonal growth to stop. The control problem to be solved in this paper is presented in the following statement. 

\textbf{Problem:} Develop a feedback control law of $q_{\rm s}(t)$ so that $l(t)$ converges to $l_{\rm s}$ for a given desired length of the axon $l_{\rm s}>0$, subject to the governing equations \eqref{sys1}--\eqref{sys5}. 

\subsection{Steady-state solution} 
To tackle the problem stated above, we first solve a steady-state solution of the concentration profile for a given axon length $l_{\rm s}$. By setting the time derivatives in \eqref{sys1}, \eqref{sys4}, and \eqref{sys5} to zero, one can derive the steady-state solution of \eqref{sys1}-\eqref{sys5} as
\begin{align}
 \label{ceq} 
c_{\rm eq}(x) = c_{\infty} \left( K_{+} e^{\lambda_+ (x - l_{\rm s})} + K_- e^{\lambda_{-} (x - l_{\rm s}) } \right),
\end{align}
where

\begin{small}
\begin{align}
    \lambda_+ =& \frac{a + \sqrt{a^2 + 4 D g}}{2 D}, \quad \lambda_- = \frac{a - \sqrt{a^2 + 4 D g}}{2 D}, \\
K_+ = & \frac{1}{2} +  \frac{a  - 2 g l_{\rm c} }{2 \sqrt{a^2 + 4 D g}},  K_- =  \frac{1}{2} -  \frac{a  - 2 g l_{\rm c} }{2 \sqrt{a^2 + 4 D g}}.
\end{align}
\end{small}

The steady-state input for the concentration flux in the soma is obtained as
\begin{align}
    q_{\rm s}^* = - c_{\infty} \left( K_{+} \lambda_+ e^{ - \lambda_+ l_{\rm s}} + K_- \lambda_- e^{ - \lambda_{-} l_{\rm s} } \right).
\end{align}

\subsection{Reference error system} 
Let $u(x,t)$, $z_1(t)$, and $z_2(t)$ be the reference error states, and $U(t)$ be the reference error input, defined as
\begin{align}
u(x,t) =& c(x,t) - c_{\rm eq}(x), \\
z_{1}(t) =& c_{\rm c}(t) - c_{\infty}, \\
z_2(t) =& l(t) - l_{\rm s}, \\
U(t) = & - ( q_{\rm s}(t) - q_{\rm s}^*). 
\end{align}
By substituting the steady-state solution in \eqref{ceq} from the governing equations \eqref{sys1}-\eqref{sys5}, the reference error system is obtained as in \cite{diehl2014one, krstic2020materials}. 

Let $X \in \R^2$ be an ODE state vector for the reference error states $z_1(t)$ and $z_2(t)$, defined by  
\begin{align} \label{xdef}
    X(t)=[ z_1(t) \quad z_2(t)]^\top . 
\end{align}
Applying the linearization of $X(t)$ around zero states leads to the following linearized reference error system (see Section 12-2 in \cite{krstic2020materials} for the detailed derivation): 
\begin{align}
    \label{ulin-PDE}
u_t (x,t) =& D u_{xx} (x,t) - a u_x (x,t) - g u(x,t) , \\
u_x(0,t) = & U(t), \label{ulin-BC1} \\
\label{linreferr3}u(l(t),t) =&C^\top X(t)  , \\
\dot{X}(t) = & A X(t) + B u_x (l(t), t), \label{ulin-ODE}
\end{align}
where 
 \begin{align} \label{AB-def} 
 A = &\left[ 
 \begin{array}{cc}
 \tilde a & 0 \\
 r_{\rm g} & 0
 \end{array}  
 \right] , \quad B =  \left[ 
 \begin{array}{c}
 - \beta \\
 0
 \end{array}  
 \right], \\
 C = &\left[1 \quad - \frac{(a-gl_{\rm c}) c_{\infty}}{D}\right]^\top  .  \label{C-def}
\end{align}

\section{Control Design} \label{sec:control} 

The control design in this paper is based on a backstepping transformation \cite{krstic2008boundary}, which maps the reference error system $(u,X)$ to a target system $(w,X)$. The backstepping transformation and the target system in this paper are given in the remainder of this section.
\subsection{Backstepping transformation and target system}
Following the procedure in \cite{koga2018control} for the Stefan problem, we consider the following backstepping transformation:

\begin{small}
\begin{align}
 \label{bkst}
w(x,t) = & u(x,t) - \int_x^{l(t)} k(x,y) u(y,t) dy 
- \phi(x - l(t))^\top X(t),  
\end{align} 
\end{small}where $k(x,y) \in \R $ and $\phi(x-l(t)) \in \R^2$ are the gain kernel functions to be determined. We suppose the desired target system as  
 \begin{align} 
\label{tar-PDE} w_t (x,t) =& D  w_{xx} (x,t) - a w_x (x,t) - g w(x,t) \notag\\
&- \dot l(t) F(x,X(t))  , \\
\label{tar-BC1} w_x(0,t) = & \gamma w(0,t),\\
\label{tar-BC2}w(l(t),t) =&0 , \\
\label{tar-ODE} \dot{X}(t) = & (A + BK^\top) X(t) + B w_x (l(t), t),
 \end{align}
where $K \in \R^2$ is a feedback control gain vector chosen to make $A +BK$ Hurwitz which means that $(A,B)$ is controllable. In detail, by setting 
 \begin{align} \label{K-def} 
 K = [ k_1 \quad k_2],  \quad k_1 > \frac{\tilde a}{\beta} , \quad k_2 > 0 ,
 \end{align} 
 one can observe that the conditions for $(k_1, k_2)$ makes $A + BK$ Hurwitz.
 Also, the redundant nonlinear term $F(x,X(t)) \in \R$ in \eqref{tar-PDE} caused by  the moving boundary is described as
 \begin{align} \label{F-def} 
     F(x,X(t))= \left(\phi'(x-l(t))^T-k(x, l(t)) C^T \right) X(t) .
 \end{align}
\subsection{Gain kernel solution}

 Taking the time and spatial derivatives of \eqref{bkst} together with the solution of \eqref{ulin-PDE}-\eqref{ulin-ODE}, and substituting $x = l(t)$ in both the transformation \eqref{bkst} and its spatial derivative, and by matching
with the target system \eqref{tar-PDE}--\eqref{tar-ODE}, we have the following PDE and an ODE for gain kernels.

\begin{align}
&k_{xx}(x,y)-k_{yy}(x,y)=\frac{a}{D}\left(k_x(x,y)+k_y(x,y)\right),\label{kernel1}\\
&k_x(x,x)+k_y(x,x)=0,\label{kernel2}\\
&k(x, l(t))=-\frac{1}{D}\phi(x-l(t))^\top B,\label{kernel3}\\
&D\phi''(x-l(t))^\top-a\phi'(x-l(t))^\top-\phi(x-l(t))^T\left[gI+A\right]\notag\\
&\ \ \ \ \ \ \ \ \ \ \  -Dk_y(x,l(t))C^\top+ak(x,l(t))C^\top=0, \label{phiODE1}\\
&\phi(0)=C,\quad \phi'(0)=k(l(t),l(t))C^\top+K^\top. \label{phiODE3}
\end{align}
By the conditions \eqref{kernel1}--\eqref{kernel3}, the solution of $k(x,y)$ is uniquely given by 
\begin{align}
    k(x,y)= -\frac{1}{D}\phi(x-y)^\top B.
    \label{kstar}
\end{align}
Substituting \eqref{kstar} into \eqref{phiODE1}--\eqref{phiODE3}, the ODE of $\phi(\cdot)$ becomes
\begin{small}
\begin{align}
    &D\phi''(x-l(t))^\top-\phi'(x-l(t))^\top \left(BC^\top+aI\right)\notag\\&\ \ \ \ \ \ \ \ \ \ \ \ \ \ \ \ -\phi(x-l(t))^\top\left[gI+A+\frac{a}{D}BC^\top\right]=0,\label{phiup1}\\
    &\phi(0)=C, \quad \phi'(0)^\top=-\frac{1}{D}C^\top BC^\top+K^\top. \label{phiup3}
\end{align}
\end{small}
The solution to \eqref{phiup1}-\eqref{phiup3} is given by (see \cite{tang2011state})
\begin{align}
    \phi(x)^\top=\begin{bmatrix}C^\top & K^\top-\frac{1}{D}C^\top BC^\top\end{bmatrix}e^{N_1x}\begin{bmatrix} I \\ 0
\end{bmatrix},
\label{phix}
\end{align}
where 
the matrix $N_1 \in \R^{4 \times 4}$ is defined as 
\begin{align}
    N_1=\begin{bmatrix}0 & \frac{1}{D}\left(gI+A+\frac{a}{D}BC^\top\right)\\ I &\frac{1}{D}\left(BC^\top+aI\right)\end{bmatrix}.
\end{align}

In order to check invertibility of \eqref{bkst}, we define the inverse Volterra Transformation
\begin{small}
\begin{align}
    u(x,t)=&w(x,t)+\int_{x}^{l(t)}q(x,y)w(y,t)dy+\varphi(x-l(t))^\top X(t) . 
\end{align}
\end{small}
Then, we apply the same strategy to obtain the inverse gain kernel functions, $q(x,y)\in \mathbb{R}$, and $\varphi(x-l(t))\in \mathbb{R}^2$. The PDE and ODE for the inverse gain kernels are
\begin{align}
    &q_{xx}(x,y)-q_{yy}(x,y)=\frac{a}{D}\left(q_x(x,y)+q_y(x,y)\right)\label{eqn:inv-ker-1}\\
    &q_x(x,x)+q_y(y,y)=0\\
    &q(x,l(t))=-\frac{1}{D}\varphi(x-l(t))^\top B \\
    &D\varphi^{''}(x-l(t))^\top+a\varphi^{'}(x-l(t))^\top\nonumber \\
    & \ \ \ \ \ \ \ \ \ \ \ \ \ \ \ +\left(gI+A+BK^\top\right)\varphi(x-l(t))^\top=0 \\
    &\varphi(0)=C, \quad \varphi'(0)=K
    \label{eqn:inv-ker-6}
\end{align}
which is well-posed, so one can obtain the solution of \eqref{eqn:inv-ker-1}-\eqref{eqn:inv-ker-6} by applying the same procedure as the one we applied to the forward kernel equations.
\subsection{Control law} 
The control design is derived from the boundary condition \eqref{tar-BC1} of the target system at $x = 0$. 
Substituting $x = 0$ into the transformation \eqref{bkst} and its spatial derivative, and using boundary conditions for both the target system and the error system, the control input is described as follows

 \begin{align}
  U(t)= & \left(\gamma-\frac{\beta}{D}\right)u(0,t) -\frac{1}{D}\int_0^{l(t)}p(x)Bu(x,t)dx\notag\\
   &+p(l(t))X(t), 
\label{real-input}
\end{align}
where 
    $p(x) = \phi'(-x)^\top-\gamma\phi(-x)^\top$.
Plugging the system matrices $(A,B,C,K)$ into the gain kernel function \eqref{phix}, and calculating the transition matrix, one can explicitly derive the function $p(x) \in \R^2$. 

\section{Proof of Main Result} \label{sec:stability} 

This section presents the main result of the paper, which is a local stability of the closed-loop system, with providing its proof by considering the $H_1$-norm 
\begin{align}
    Z(t) = || u(\cdot, t) ||^2 + || u_x(\cdot, t) ||^2 + X^\top X. 
\end{align}
We present our main theorem below. 


\begin{theorem}
Consider the closed-loop system consisting of the plant \eqref{ulin-PDE}--\eqref{ulin-ODE} with the control law \eqref{real-input} and \eqref{phix}. Suppose the control parameter $\gamma >0$ is chosen to satisfy  $\gamma \geq \frac{a}{D}$. Then, there exist positive parameters $ \bar M>0$, $c>0$, and $\kappa>0$, such that if $Z(0)< \bar M$ then the following norm estimate holds
\begin{align}
    Z(t)\leq c Z(0) \exp( - \kappa t) 
\end{align} 
for all $ t\geq 0$, which guarantees the local exponential stability of the closed-loop system. 
\end{theorem}


\subsection{Basic idea of proof of Theorem 1}
Since the backstepping transformation \eqref{bkst} is invertible, the stability property of the target system \eqref{tar-PDE}-\eqref{tar-ODE} is equivalent to the stability property of the closed-loop system consisting of the plant \eqref{ulin-PDE}-\eqref{ulin-ODE} with the control law \eqref{real-input}. 
The local stability of the target system is studied under the assumption that the following two conditions hold:
\begin{align} \label{ineq-l} 
     0 < l(t) \leq \bar l, \\
    |\dot l(t) | \leq \bar v, \label{ineq-ldot}
\end{align}
for some $\bar l>l_{\rm s} >0$ and $\bar v>0$. The restricted initial state will be given later in order to satisfy these conditions for all $t \geq 0$. 
\subsection{Useful inequalities}
The following inequalities are used in Lyapunov analysis. Under the condition \eqref{ineq-l}, and by boundary conditions \eqref{tar-BC1} and \eqref{tar-BC2}, Poincare's inequality is provided as
\begin{align} \label{poincare-1} 
    || w ||^2 \leq &  4 \bar{l}^2 ||w_{x}||^2, \quad
    || w_{x} ||^2 \leq & 2 \bar l \gamma^2 w(0,t)^2 + 4 \bar{l}^2 ||w_{xx} ||^2, 
\end{align}
and the Agmon's inequality is given as 
\begin{align}
  \label{Agmon} 
    w_x(l(t),t)^2 \leq & 2 \gamma^2 w(0,t)^2 + 4 \bar l ||w_{xx}||^2 . 
\end{align}




\subsection{Proof of Lyapunov stability}
We consider the Lyapunov function of the target system as
\begin{align}
    V = d_1 V_1 + V_2 + \fr{\gamma}{2} w(0,t)^2 + d_2 V_3, 
    \label{Vtotal}
\end{align}
where $d_1>0$ and $d_2 >0$, and each Lyapunov function is
\begin{align}
    V_1 =& \fr{1}{2} ||w||^2 := \fr{1}{2} \int_0^{l(t)} w(x,t)^2 dx, \\
\label{V2-def}    V_2 =& \fr{1}{2} ||w_x||^2 := \fr{1}{2} \int_0^{l(t)} w_x(x,t)^2 dx, \\
V_3=&X(t)^\top PX(t),
\end{align}
where positive definite matrix $P$ is the solution of the following Lyapunov equation
\begin{align}
    (A + BK^\top )^\top P + P (A + BK^\top ) = - Q,
\end{align}
for some positive definite matrix $Q$. Since $(A+BK^\top)$ is Hurwitz, positive definite matrices $P$ and $Q$ exist. Due to the positive definiteness, it holds that
\begin{align} \label{ineq-XPX} 
    \lam_{\rm min}(P) X^\top X \leq X^\top P X \leq \lam_{\rm max}(P) X^\top X, 
\end{align}
where $\lam_{\rm min}(P) >0$ and $\lam_{\rm max}(P)>0$ are the smallest and the largest eigenvalues of $P>0$. Then we have the following lemma. 

\begin{lemma} Assume that \eqref{ineq-l}--\eqref{ineq-ldot} are satisfied with \begin{align}
    \bar{v}=\min\left\{\frac{g }{4\gamma}, \frac{D}{8\bar{l}}\right\}, 
\end{align}
for all $ t \geq 0$. Then, for sufficiently large $d_1>0$ and small $d_2>0$, there exists a positive constant $\beta>0$ such the following norm estimate holds for all $t \geq 0$:
\begin{align}
    \dot{V}\leq -\alpha V+\beta V^{3/2}, 
    \label{vdotbound}
\end{align}
where $\alpha=\min\left\{2g+\frac{D}{4\bar{l}},\frac{4g+d_1D}{2},\frac{\lambda_{\rm min}(Q)}{2\lambda_{\rm max}(P)}, \frac{d_2\left(2d_1D+g\right)}{4}\right\}$. 
\end{lemma} 

\begin{proof} 
Taking the time derivative of the Lyapunov functions along the target system \eqref{tar-PDE}--\eqref{tar-ODE}, we have
\begin{align}
\dot V_1 =&  - D || w_{x}||^2  - g ||w||^2 - \left(\gamma D - \frac{a}{2} \right) w(0,t)^2 \notag\\
&+ \dot l(t) \int_0^{l(t)} F(x,X(t)) w(x,t) dx. \label{V1dot-2}\\
    \dot{V}_2 =& - D || w_{xx}||^2 + a  \int_0^{l(t)} w_{xx}(x,t) w_{x}(x,t) dx  \notag\\&- g || w_{x}||^2 
    - \gamma w(0,t) w_{t}(0,t) -\frac{1}{2}\dot{l}(t)w_x(l(t),t)^2\notag\\
& - \dot{l}(t) (F(l(t),X(t)) w_x(l(t),t)- \gamma F(0,X(t)) w(0,t)) \notag\\
- &\gamma g w(0,t)^2-  \dot{l}(t)\int_0^{l(t)} F_x(x,X(t)) w_{x}(x,t)  dx  \label{A-3-2:V2dot2-1} \\
\dot{V}_3=&-X(t)^\top QX(t)+ 2 w_x(l(t), t)B^\top PX(t) . \label{V3dot} 
\end{align}
Applying the Agmon's inequality \eqref{Agmon} and Young's inequality to $\dot{V}_2$ in \eqref{A-3-2:V2dot2-1} and $\dot{V}_3$ in \eqref{V3dot} leads to 
\begin{align}
\dot{V}_2 \leq & - \fr{D}{4} || w_{xx}||^2 - \fr{\gamma g}{2} w(0,t)^2 - \left(g - \frac{a^2}{D} \right)  || w_{x}||^2  \notag\\
& + \dot{l}(t) \gamma F(0,X(t)) w(0,t)\notag\\&+\dot{l}(t) \int_0^{l(t)} F_x(x,X(t)) w_{x}(x,t)  dx \notag\\
&+\frac{\big|\dot{l}(t)\big|}{2}F(l(t),X(t))^2- \gamma w(0,t) w_{t}(0,t), \label{A-3-2:V2dot2-4}\\
    \dot{V}_3  \leq& -\frac{\lambda_{\rm min}(Q)}{2}X^\top X \notag\\
    & + \frac{2\big|B^\top P \big|^2}{\lambda_{\rm min}(Q)}\left(2 \gamma^2 w(0,t)^2 + 4 \bar l ||w_{xx}||^2\right).
    \label{dotV3last}
\end{align}
By using \eqref{V1dot-2}, \eqref{A-3-2:V2dot2-4} and \eqref{dotV3last}, recalling $\gamma \geq \frac{a}{D}$, and choosing the constants $d_1$ and $d_2$ to satisfy
\begin{align}
    d_1\geq \frac{2a^2}{D^2}, \
    d_2\leq \min\left\{\frac{D\lambda_{\rm min}(Q)}{64\bar{l}\big|B^TP\big|^2},\frac{g D\lambda_{\rm min}(Q)}{16a\left|B^TP\right|^2}\right\},  
\end{align}
the time derivative of Lyapunov function \eqref{Vtotal} for the target system is shown to satisfy the following inequality 

\begin{small}
\begin{align}
    \dot V &     
\leq  -\alpha V +\frac{\big|\dot{l}(t)\big|}{2}F(l(t),X(t))^2\notag+\dot{l}(t) \gamma F(0,X(t)) w(0,t)\notag\\
& + \dot{l}(t) \int_0^{l(t)} F_x(x,X(t)) w_{x}(x,t) + d_1  F(x,X(t)) w(x,t) dx. 
    \label{dotVwalpha}
\end{align}
\end{small}



Taking the square of \eqref{F-def}, it follows that the redundant nonlinear terms in \eqref{dotVwalpha} can be bounded by a quadratic norm of the ODE state. Namely, there exist positive constants $L_1>0$, $L_2>0$, and $L_3$ such that $F(0,X(t))^2 \leq L_1 X^\top X$, $F(l(t), X(t))^2 \leq L_1 X^\top X$, $ \int_0^{l(t)} F_x(x,X(t))^2 dx \leq L_2 X^\top X$, $\int_0^{l(t)} F(x,X(t))^2 dx \leq L_3 X^\top X$. In addition, 
    $\dot{l}(t)$ can be rewritten as $\dot{l}(t)=r_{\rm g} e_1 X$. 
By using these relations and applying Cauchy-Schwarz inequality and Young's inequality, one can show that 


\begin{small}
\begin{align}
    \dot V &\leq  -\alpha V +r_{\rm g}\left( \frac{L_1(1+\gamma)+1}{2d_2\lambda_{\rm min}(P)}+(1+L_3+d_1L_2)\right)V^{3/2}. 
\end{align}
\end{small}
Thus, there exists $\beta>0$ such that Lemma 1 holds.
\end{proof}
\subsection{Guaranteeing the conditions for all time}

Next, we prove important lemmas to conclude with Theorem 1 ensuring the local stability of the closed-loop system.
\begin{lemma}
    There exists a positive constant $M_1>0$ such that in the region $\Omega_1 := \{(w,X) \in H_1 \times \R^2 | V(t) < M_1\}$ the conditions \eqref{ineq-l} and \eqref{ineq-ldot} are satisfied. 
\end{lemma} 
\begin{proof} By using \eqref{xdef} and \eqref{sys5}, $X(t)$ can be described as
    $X(t)=\left[\begin{array}{cc}
\fr{\dot l(t)}{r_{\rm g}}&
l(t)-l_{\rm s}
    \end{array}\right]^\top$.
Thus, for some $r>0$, if $| X|<r$ then both the following two inequalities  hold:
\begin{align}
    \left|\frac{\dot{l}(t)}{r_{\rm g}}\right|< r, \ \ \big|l(t)-l_{\rm s}\big|< r. 
\end{align}
The former inequality tells that if $r < \frac{\bar v}{r_{\rm g}}$ then the condition \eqref{ineq-l} holds. Moreover, the latter inequality yields $- r + l_{\rm s} < l(t) < r + l_{\rm s}$, and thus if both $r < l_{\rm s}$ and $r < \bar l - l_{\rm s}$ hold, then the condition \eqref{ineq-ldot} holds.   
Therefore, the constant $r>0$ is chosen as  
$
        r=\min \left\{\frac{\bar{v}}{r_{\rm g}}, l_{\rm s}, \bar{l}-l_{\rm s}\right\}.
$
Since $|X|^2 \leq \frac{1}{\lam_{\rm min}(P)} X^\top P X \leq \fr{d_2}{\lam_{\rm min}(P)} V$, we deduce that by setting $M_1 = \fr{\lam_{\rm min}(P)}{d_2} r^2$, if $V(t) < M_1 $ holds then $| X| < r$ and thus the conditions \eqref{ineq-l} and \eqref{ineq-ldot} are satisfied, by which we can conclude Lemma 2. 
\end{proof} 

\begin{lemma}
There exists a positive constant $M>0$ such that if $V(0) < M$ then the conditions \eqref{ineq-l} and \eqref{ineq-ldot} are satisfied and the following norm estimate holds: 
\begin{align} \label{V-decay} 
    V(t) \leq V(0) \exp\left(- \fr{\alpha}{2} t \right).  
\end{align}
\end{lemma}
\begin{proof}
For a positive constant $M>0$, let $\Omega := \{ (w,X) \in H_1 \times \R^2 | V(t) < M\}$. By Lemma 2, it is easily shown that if $M \leq M_1$ then $\Omega \subset \Omega_1$, and thus the conditions \eqref{ineq-l} and \eqref{ineq-ldot} are satisfied in the region $\Omega$. Thus, by Lemma 1, the norm estimate \eqref{vdotbound} holds. Moreover, by setting $M \leq \fr{\alpha^2}{4 \beta^2}$, we can see that applying $V(t) < M$ to \eqref{vdotbound} leads to 
\begin{align}
    \dot V \leq - \fr{\alpha}{2} V, 
\end{align}
by which the norm estimate \eqref{V-decay} is deduced. Since \eqref{V-decay} is a monotonically decreasing function in time, by setting $M = \min\{M_1, \fr{\alpha^2}{4 \beta^2}\} $, the region $\Omega$ is shown to be an invariant set. Thus, if $V(0) < M$, then $V(t) < M$ for all $t \geq 0$, and one can conclude with Lemma 3. 
\end{proof}

Finally, owing to Lemma 3, and the equivalent norm estimate in the $H_1$-norm between the target and the closed-loop system, the local stability of the closed-loop system is proved, which completes the proof of Theorem 1.

\section{Numerical Simulation}
\begin{figure}[t]
    \centering
    \subfloat[\textbf{The axon length converges to the desired length. }\label{subfig-2}]{%
\includegraphics[width=0.89\linewidth]{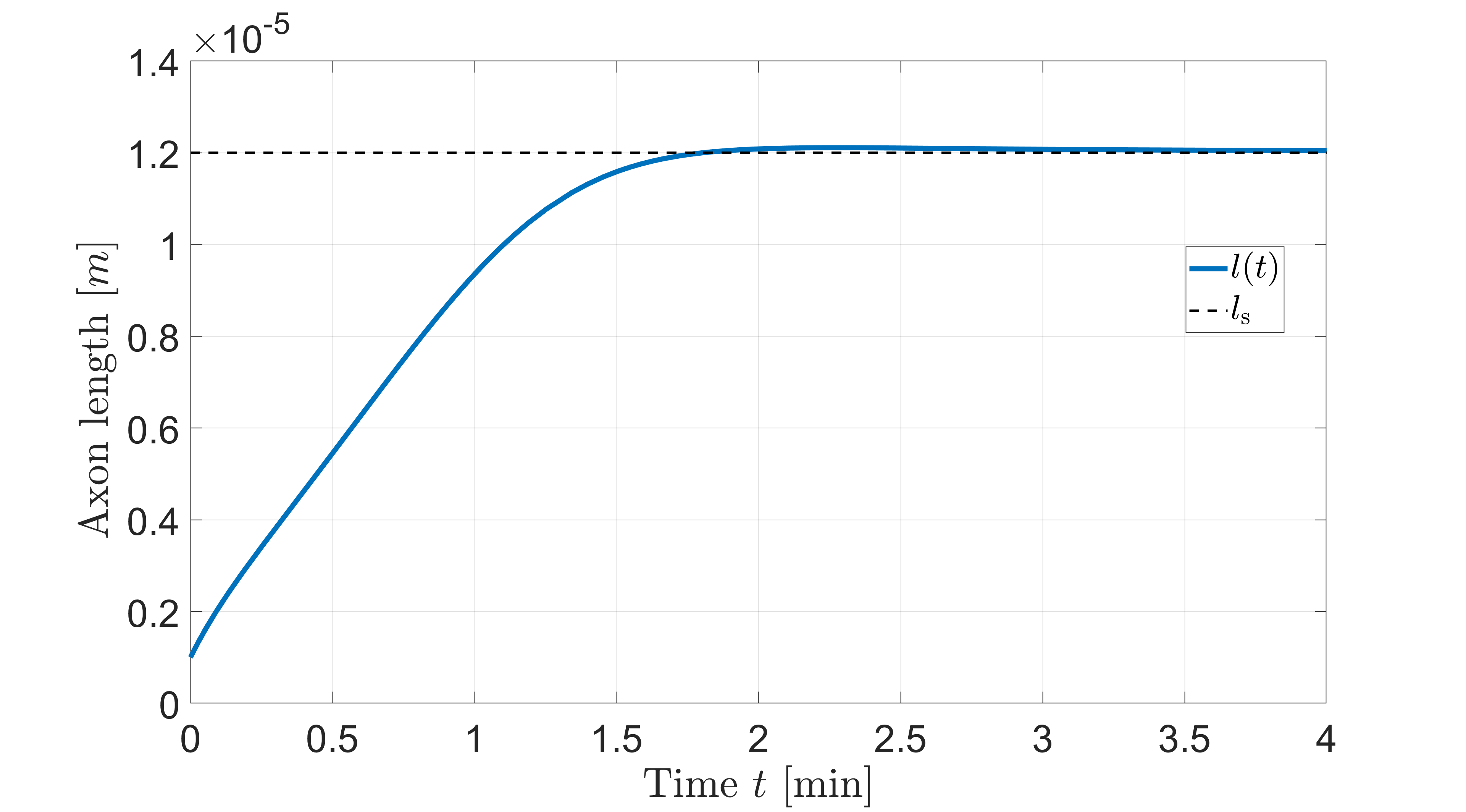}}
    \vfill
\subfloat[\textbf{The tubulin concentration converges to the steady-state. }\label{subfig-1}]{ \includegraphics[width=0.90\linewidth]{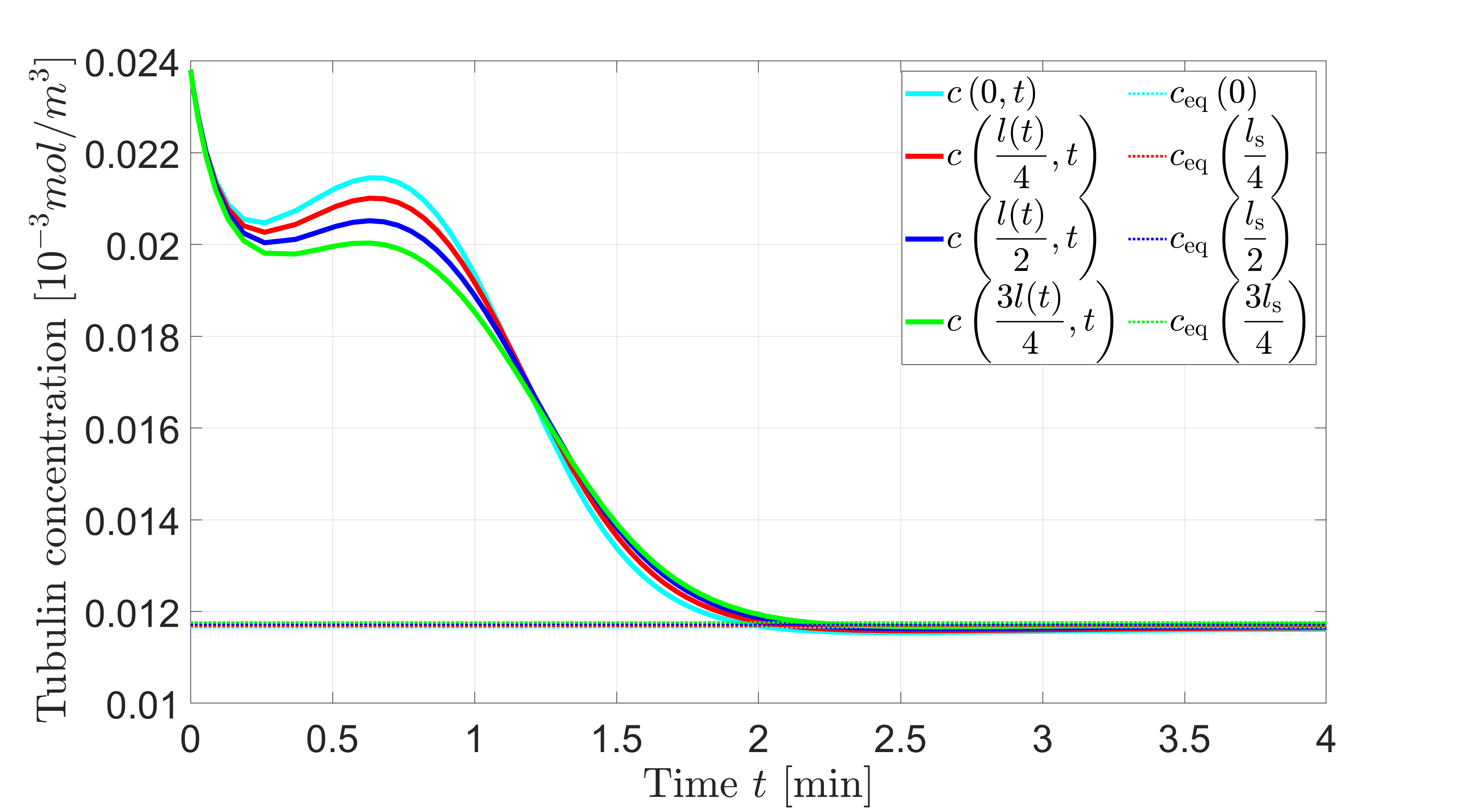}
}
    \caption{Closed-loop response under the designed control law. }
\end{figure}

\begin{table}[t]
\hfill
	\caption{\label{tab:initial}Biological constants and control parameters}
	\centering
	\begin{tabular}{cccc} 
		\hline
		Parameter  & Value  & Parameter & Value \\
		\hline
		$D$ & $10\times10^{-6}  m^2/s$& $\gamma$ &  $10^4$\\
		$a$& $1\times 10^{-8}  m/s$ & $l_{\rm c}$ & $4\mu m$\\
		$g$& $5\times 10^{-7} \ s^{-1}$ & $l_s$ & $12\mu m$\\
		$r_{\rm g}$& $1.783\times 10^{-5} \ m^4/(mol s)$ & $l_0$& $1\mu m$ \\
		$c_{\infty}$ &  $0.0119  \ mol/m^3$ & $k_1$& $-0.1$\\
		$\tilde{r}_{\rm g}$ & $0.053$ & $k_2$ & $10^{13}$ \\
 \hline
	\end{tabular}
\end{table} 
Simulation is performed for the axon growth dynamics \eqref{sys1}-\eqref{sys5} under the designed control law \eqref{real-input}. We use the biological constants given in \cite{diehl2014one}, and choose the control parameters, as shown in Table \ref{tab:initial}. The initial tubulin concentration is set as a constant profile $c_0(x) =2c_{\infty}$, and the initial axon length is set as $l_0=1 \mu m$.

Fig. \ref{subfig-2} shows that the axon length converges to the desired length $l_{\rm s}$. In addition, Fig. \ref{subfig-1} depicts that the tubulin concentration converges to the steady-state solution. Hence, we observe that the simulation result is consistent with our theoretical results.

\section{Conclusion} \label{sec:conclusion} 
In this paper, a boundary feedback control for an axonal growth model governed by a coupled PDE-ODE with a moving boundary is studied. The backstepping transformation is utilized for the original system to convert it to the target system which has a stable system matrix in ODE. The gain kernel solutions in the transformation are obtained, and the boundary feedback control law is designed explicitly. Using Lyapunov's method, we first prove that the target system is locally exponentially stable in the spatial $H_1$-norm and then we prove local stability of
the original coupled PDE-ODE system of axonal growth model. Showing the local stability without applying linearization will be studied in future. 


\bibliographystyle{IEEEtran}
\bibliography{BIB_CDC21.bib}

\end{document}